\documentclass[11pt]{article}
\usepackage[left=3.5cm,right=3.5cm,top=3cm,bottom=3cm,a4paper]{geometry}

\usepackage{amsthm}
\usepackage{mathrsfs}

\usepackage{graphicx}
\usepackage{amsmath}
\usepackage{amsfonts}
\usepackage{amssymb}

\usepackage{listings,xcolor}

\usepackage{hyperref}

\usepackage{graphics, setspace}

\newcommand{\mathsym}[1]{{}}
\newcommand{\unicode}[1]{{}}

\usepackage{verbatim}
\usepackage[mathlines]{lineno}

\usepackage{enumerate}
\usepackage{enumitem}

\usepackage[labelsep=colon]{caption}

\definecolor{red}{rgb}{1,0,0}

\date{}

\usepackage{tikz}
\usetikzlibrary{matrix,backgrounds,fit,trees}
\usetikzlibrary{through,calc}

\newtheorem{thm}{Theorem}[section]

\newtheorem{prop}[thm]{Proposition}
\newtheorem{cor}[thm]{Corollary}

\newtheorem{conj}[thm]{Conjecture}

\def\noi{\noindent}

\newcommand{\abs}[1]{\left\vert#1\right\vert}
\newcommand{\set}[1]{\left[#1\right]}

\begin{document}

\title{Sharp spectral bounds for the edge-connectivity of a regular graph}
\author{Suil O \thanks{Applied Mathematics and Statistics, The State University of New York Korea, Incheon, 21985, Republic of Korea (suil.o@sunykorea.ac.kr). Research supported by NRF-2017R1D1A1B03031758.}
\and
Jeong Rye Park \thanks{Finance.Fishery.Manufacture Industrial Mathematics Center on Big Data, 
	Pusan National University, Busan, 46241, Republic of Korea (parkjr@pusan.ac.kr). Research supported by NRF-2017R1A5A1015722.}
\and
Jongyook Park\thanks{Department of Applied Mathematics, Wonkwang University, Iksan, Jeonbuk, 54538,  Republic of Korea (jongyook@wku.ac.kr). Research supported by NRF-2017R1D1A1B03032016.}
\and
Hyunju Yu\thanks{Department of Mathematics, Kyungpook National University, Daegu, 41566, Republic of Korea (lojs4110@gmail.com).}
}


\maketitle

\begin{abstract}
Let $\lambda_2(G)$ and $\kappa'(G)$ be the second largest eigenvalue
and the edge-connectivity of a graph $G$, respectively.
Let $d$ be a positive integer at least 3. For $t=1$ or 2,
Cioab{\v a} proved sharp upper bounds for $\lambda_2(G)$ in
a $d$-regular simple graph $G$ to guarantee that $\kappa'(G) \ge t+1$.
In this paper, we settle down for all $t \ge 3$.
\end{abstract}

\noi{\bf Keywords.}
Second largest eigenvalue;
edge-connectivity;
simple regular graphs;
algebraic connectivity.\\

\noi{\bf AMS subject classifications.}
05C50, 05C40.

\section{Introduction}
A {\it multigraph} is a graph that can have multiple edges but does not contain loops. A {\it simple graph} is a graph without loops or multiple edges. Thus a simple grpah is a special case of a multigraph. Let $V(G)$ and $E(G)$ be the vertex and edge set of $G$, respectively.  For $S\subseteq V(G)$, we denote by $G[S]$ and by $G-S$ the graph induced by $S$ and the subgraph of $G$ obtained from $G$ by deleting the vertices in $S$ together with the edges incident to vertices in $S$, respectively.
 A multigraph $G$ is {\em $k$-vertex-connected} if $|V(G)|>k$ and for $S\subseteq V(G)$ with $|S|<k$, $G-S$ is connected.
 The {\em vertex-connectivity} of $G$, written $\kappa(G)$, is the maximum $k$ such that $G$ is $k$-vertex-connected.  The {\em adjacency matrix} $A(G)$ of a  multigraph $G$ is the matrix whose rows
and columns are indexed by the vertex set $V(G)$, and the ($u, v$)-entry is the number of edges between $u$ and $v$. The {\em eigenvalues} of $G$ are the eigenvalues of $A(G)$. We denote the eigenvalues of $G$ by $\lambda_1(G), \ldots, \lambda_{|V(G)|}(G)$, indexed in non-increasing order. For $v \in V(G)$, the {\em degree} of $v$ is the number of edges incidient to $v$. If every vertex in $V(G)$ has the same degree $d$, then $G$ is called {\em $d$-regular}. The {\em Laplacian matrix} of $G$ is $L(G)=D(G)-A(G)$, where $D(G)$ is the diagonal matrix of degrees. We denote the eigenvalues of $L(G)$  by $\mu_1(G), \ldots, \mu_{|V(G)|}(G)$, indexed in non-decreasing order. Note that $\mu_2(G)>0=\mu_1(G)$ for a connected graph $G$, and that if $G$ is $d$-regular, then $\lambda_i(G)=d-\mu_i(G)$  for all $i\in \{1,\ldots,|V(G)|\}$. 

In 1973, Fiedler~\cite{Fiedler} showed that for any non-complete simple graph $G$, we have $\mu_2(G) \le \kappa(G)$, and he called $\mu_2(G)$ the {\em algebraic connectivity} of $G$. His work stimulated a lot of research in graph theory (see, for example ~\cite{abreu, Survey with Cheeger ineq., K & Sudakov book chapter, Mohar, O- alg conn mult graphs, rad, wu}). In 2002, Kirkland, Molitierno, Neumann and Shader~\cite{kirkland} characterized when $\mu_2(G)=\kappa(G)$.
The first author~\cite{O- alg conn mult graphs} extened the Fiedler's result to multigraphs.

A multigraph $G$ is {\em $t$-edge-connected} if 
for $S \subseteq E(G)$ with $|S|<t$, $G-S$ is connected. The {\em edge-connectivity} of $G$, written $\kappa'(G)$, is the maximum $t$ such that $G$ is $t$-edge-connected. Note that $\kappa(G)\leq\kappa'(G)$. In 2004, Chandran~\cite{Chandran 2004} showed that if $G$ is an $n$-vertex $d$-regular simple graph with $\lambda_2(G)<d-1-\frac{d}{n-d}$, then $\kappa'(G)=d$, and every minimum diconnecting edge set is trivial. In 2006, by Krivelevich and Sudakov~\cite{K & Sudakov book chapter}, this result was slightly improved as follows: if $G$ is a $d$-regular simple graph with $\lambda_2(G)\leq d-2$, then $\kappa'(G)=d$. In 2010, Cioab{\v a}~\cite{Cioaba 2010} proved that if $\lambda_2(G)<d-\frac{2t}{d+1}$, then $\kappa'(G)\geq t+1$, and he strengthened the result for $t=1,2$.

\begin{thm}\label{Cioaba1}\cite{Cioaba 2010}
Let $d$ be an odd integer at least $3$ and let $\pi(d)$ be the largest root of $x^3-(d-3)x^2-(3d-2)x-2=0$. If $G$ is a $d$-regular simple graph such that $\lambda_2(G)<\pi(d)$, then $\kappa'(G)\geq2$.
\end{thm}

\begin{thm}\label{Cioaba2}\cite{Cioaba 2010}
If $G$ is a $d$-regular simple graph such that $\lambda_2<\frac{d-3+\sqrt{(d+3)^2-16}}{2}$, then $\kappa'(G)\geq3$.
\end{thm}

Note that the upper bounds for $\lambda_2(G)$ in Theorem~\ref{Cioaba1} and~\ref{Cioaba2} are sharp (see, for example~\cite{Cioaba 2010}) and that the first author~\cite{O-thesis} conjectured in his Ph.D thesis that it can be generalized for all $t\geq3$. In this paper, we settle down the conjecture (see, Theorem~\ref{main}) in a positive way.

\begin{conj}\label{O-conj}\cite{O-thesis}
Let  $\rho(d,t)=\begin{cases}
\frac{d-4+\sqrt{(d+4)^2-8t}}{2}& \text{ when } t \text{ is odd,}\\
\frac{d-3+\sqrt{(d+3)^2-8t}}{2}& \text{ when } t \text{ is even. }
\end{cases}$

For $t\geq3$, if $G$ is a $d$-regular simple graph such that $\lambda_2(G)<\rho(d,t)$, then $\kappa'(G)\geq t+1$.
\end{conj}

The first author~\cite{O-Edge-conn from eigenvalues} proved sharp upper bounds for $\lambda_2(G)$ in a $d$-regular multigraph. With Abiad, Brimkov, Martinez-Rivera, and Zhang, the first author~\cite{ABMOZ} also gave an upper bound for $\lambda_2(G)$ in a $d$-regular multigraph with given order.

This paper is organized as follows: after Introduction, in Section 2, we describe examples of graphs which assure that the bounds for $\lambda_2$ in Conjecture~\ref{O-conj} are sharp. In Section 3, we give a proof of Conjecture~\ref{O-conj}.

For underfined terms, see West~\cite{west} or Godsil and Rolye~\cite{GR}.

\section{Construction}
In this section, we provide $d$-regular simple graphs $G_{d,t}$ with $\lambda_2(G_{d,t})=\rho(d,t)$ and with $\kappa'(G_{d,t})=t$, where $$\rho(d,t)=\begin{cases}
\frac{d-4+\sqrt{(d+4)^2-8t}}{2}& \text{ if } t \text{ is odd, }\\
\frac{d-3+\sqrt{(d+3)^2-8t}}{2}& \text{ if } t \text{ is even.}
\end{cases}$$
Those graphs show that the bounds for $\lambda_2(G)$ in Conjecture~\ref{O-conj} are sharp if the conjecture is true. 
Note that for a $d$-regular graph $G$ and odd $t$, if $\kappa'(G)=t$, then $d$ must be odd by the Degree-Sum formula.

We denote the \emph{complete graph} and the \emph{cycle} on $n$ vertices by $K_n$ and $C_n$, respectively.
The join of two graphs $G$ and $H$, written $G \vee H$, is the graph obtained from $G$ and $H$ by joining the vertices of $G$ and $H$. The complement of a simple graph $G$, denoted $\overline{G}$, is the graph with the vertex set $V(G)$ defined by $uv \in E(\overline{G})$ if and only if $uv \notin E(G)$.

Suppose that $d$ and $t$ are positive integers such that $3 \le t \le d-1$. 
Let $$H_{d,t}=\begin{cases}
{\overline{\frac{d+2-t}{2}K_2}} \vee \overline{C_t} & \text{ if } t \text{ is odd, }\\
K_{d+1-t} \vee {\overline{\frac{t}{2} K_2}} & \text{ if } t \text{ is even. }
\end{cases}$$

Let $G_{d,t}$ be the graph obtained from two copies of $H_{d,t}$ by adding $t$ edges
so that the resulting graph is $d$-regular. We first show that $\kappa'(G_{d,t})=t$ by using Proposition~\ref{prop}. 

For $S, T \subseteq V(G)$, we denote by $[S,T]$ the set of edges joining $S$ and $T$.
\begin{prop}\label{prop}
 If $G$ is an $n$-vertex connected graph such that $\frac n2  \le \delta(G) \le n$, then $\kappa'(G) = \delta(G)$, where $\delta(G)$ is the minimum degree of $G$.
\end{prop}
\begin{proof} Consider an edge-cut $[S, \overline{S}]$ for $S\subseteq V(G)$ such that $|S| \le \frac n2$. Then $$|[S, \overline{S}]| \ge |S|(\delta(G) + 1 - |S|) \ge \delta(G).$$
\end{proof} 
\begin{thm}
For $3 \le t \le d-1$, we have $\kappa'(G_{d,t})=t$.
\end{thm}
\begin{proof} We prove the theorem for odd $t$. The proof of the other case is similar and is omitted.
	
	By the construction of $G_{d,t}$, there exist $t$ edges between two copies of $H_{d,t}$, which is a subgraph of $G_{d,t}$. Thus we have $\kappa'(G_{d,t}) \le t$.

By Proposition~\ref{prop}, we have $\kappa'(H_{d,t})=d-1 \ge t$. Since there are $t$ edges between two copies of $H_{d,t}$ in $G_{d,t}$,  for any pair of two vertices in $V(G_{d,t})$, there exists at least $t$ edge-disjoint paths between them. Thus we have the desired result.
\end{proof}
Now, we determine the second largest eigenvalue of $G_{d,t}$. Before determining it,
we introduce an important tool, which is called ``eigenvalue interlacing''. 

For an $n \times n$ matrix $A$, $B$ is a {\it principal submatrix of $A$} if $B$ is a square matrix obtained by removing the same set of rows and columns of $A$. Given two sequences of real numbers $a_1 \ge \cdots \ge a_n$ and $b_1 \ge \cdots \ge b_m$ with $m < n$, we say that the second sequence {\it interlaces} the first sequence whenever $a_i \ge b_i \ge a_{n-m+i}$ for $i = 1, \ldots , m.$

\begin{thm}{\rm \cite[Interlacing Theorem]{BH}}\label{eig_interlace} 
	If $A$ is a real symmetric $n\times n$ matrix and $B$ is a principal submatrix of $A$ of order $m \times m$ with $m < n$, then for $1 \le i \le m$, $\lambda_i(A) \ge \lambda_i(B) \ge \lambda_{n-m+i}(A)$, i.e., the eigenvalues of $B$ interlace the eigenvalues of $A$.
\end{thm}

Let $P = \{V_1, \ldots, V_s\}$ be a partition of the vertex set of a multigraph $G$
into $s$ non-empty subsets. The {\it quotient matrix $Q$} corresponding to $P$ is the $s\times s$ matrix whose entry $Q_{i,j} (1 \le i, j \le s)$ is the average number of incident edges in $V_j$ of the vertices in $V_i$. More precisely, $Q_{i,j} = \frac{|[V_i,V_j]|}{|V_i|}$ if $i \neq j$, and $Q_{i,i} = \frac{2|E(G[Vi])|}{|Vi|}$. Note that for a simple graph, $Q_{i,j}$ is just the average number of neighbors in $V_j$ of the vertices in $V_i$.

\begin{cor}{\rm \cite[Quotient Interlacing Theorem]{BH}}\label{quotient_eig_interlace}
	The eigenvalues of the quotient matrix interlace the eigenvalues of $G$.
\end{cor}

A partition $P$ is {\it equitable} if for each $1 \le i, j \le s$, any vertex $v \in V_i$ has exactly $Q_{i,j}$ neighbors in $V_j$. In this case, the eigenvalues of the quotient matrix are eigenvalues of $G$ and the spectral radius of the quotient matrix equals the spectral radius of $G$ (see\cite{BH,GR} for more details).

\begin{thm}\label{ex_bound}
For $3 \le t \le d-2$, we have $$\lambda_2(G_{d,t})=\begin{cases}
\frac{d-4+\sqrt{(d+4)^2-8t}}{2}& \text{ if } t \text{ is odd, }\\
\frac{d-3+\sqrt{(d+3)^2-8t}}{2}& \text{ if } t \text{ is even.}
\end{cases}$$
\end{thm}
\begin{proof}We prove the theorem for odd $t$. The proof of the other case is similar and is omitted.
	
Let $A$ and $D$ be the two copies of $\overline{\frac{d+2 - t}{2}K_2}$  and let $B$ and  $C$ be the two copies of $\overline{C_t}$ in $G_{d,t}$. 
Note that the spectrum of $A$ is $\{(d-t)^1, 0 ^{s_1}, (-2)^{s_2}
\}$ for some $s_1$ and $s_2$ with $s_1+s_2=d+1-t$, and that the spectrum of $B$ is included in the interval $[-3, t-3]$
since if $G$ is a $d$-regular graph, then the complement of $G$ is $(n-1-d)$-regular, and
the remaining $n -1$ eigenvalues of $\overline{G}$ are $-1 - \lambda(G)$, where $\lambda(G)$ runs through the $n - 1$ eigenvalues of $G$ belonging to an eigenvector orthogonal to the all 1's vector.

For a nontrivial engenvalue $\lambda$ of $A$ with the eigenvector $\mathbf{x}^T$, $(\mathbf{x}, \mathbf{0}, \mathbf{0}, \mathbf{0})^T$ and $(\mathbf{0}, \mathbf{0}, \mathbf{0}, \mathbf{x})^T$ are eigenvectors of $G_{d,t}$ with eigenvalue $\lambda$. And for a nontrivial eigenvalue $\lambda'$ of the induced subgraph on $B$ with eigenvector $\mathbf{y}^T$, $(\mathbf{0}, \mathbf{y}, \mathbf{y}, \mathbf{0})^T$ and $(\mathbf{0}, \mathbf{y}, -\mathbf{y}, \mathbf{0})^T$ are eigenvectors of $G_{d,t}$ with eigenvalue $\lambda' + 1$ and $\lambda' - 1$, respectively. Note that $-4 \leq \lambda' \pm 1 \leq 2$.

In order to calculate remaing 4 eigenvalues of $G_{d,t}$, we consider the quotient matrix with respect to $\{V(A), V(B), V(C), V(D)\}$, that is $$\begin{pmatrix} d-t & t & 0 & 0 \\ d+2-t & t-3 & 1 & 0 \\ 0 & 1 & t-3 & d+2 - t \\ 0& 0& t & d-t \end{pmatrix}.$$

Note that the vertex partition $\{V(A), V(B), V(C), V(D)\}$ is equitable. Thus the eigenvalues of the quotient matrix are the eigenvalues of $G$.  Since $d,$ $-2$, and $\frac{d-4 \pm \sqrt{(d+4)^2 - 8t}}{2}$ are the eigenvalues of $Q$, we conclude that $\lambda_2(G_{d,t})= \frac{d-4 + \sqrt{(d+4)^2 - 8t}}{2}$.
\end{proof}

\section{Main Results} 
\label{sec:main_results}
In this section, we show that Conjecture~\ref{O-conj} is true by using Corollary~\ref{quotient_eig_interlace} and Theorem~\ref{DRG-BCN}.
  
Let $M_n(F)$ be the set of all $n$ by $n$ matrices over a field $F$.
A matrix $A=[a_{ij}]\in M_n(F)$ is {\it tridiagonal} if 
$a_{ij}=0$, whenever $|i-j|>1$.

\begin{thm}\cite{BCN,horn}\label{DRG-BCN}
	Let $A$ be a non-negative tridiagonal matrix as follows:
	\[A=\left(
		\begin{array}{ccccc}
			a_0 & b_0    &         &         & 0       \\
			c_1 & a_1    & b_1     &         &         \\
			    & \ddots & \ddots  & \ddots  &         \\
	    		& 		 & \ddots  & \ddots  & b_{n-1} \\
			0   &        & 	       & c_n     & a_n
		\end{array} \right)\]
	Assume that each row sum of $A$ equals $d$.
	If $A$ has eigenvalues $\lambda_1, \ldots, \lambda_n$ indexed in non-increasing order, 
	then the $(n-1)\times (n-1)$ matrix 
	\[\widetilde{A}=\left(
		\begin{array}{ccccc}
			d-b_0-c_1 & b_1       &         &         & 0       \\
			c_1       & d-b_1-c_2 & b_2     &         &         \\
			          & c_2       & \ddots  & \ddots  &         \\
			          &           & \ddots  & \ddots  & b_{n-1} \\
			0         & 		  &         & c_{n-1} & d-b_{n-1}-c_{n}
		\end{array} \right)\]
	has eigenvalues $\lambda_2, \lambda_3,\ldots, \lambda_n$.
\end{thm}

Note that $\lambda_1(A)=d$ and $\lambda_2(A)=\lambda_1(\widetilde{A})$.

\begin{prop}\label{component_size}
	Let $G$ be a $d$-regular graph with $\kappa'(G)=r \le d-1$. 
If $|[S,\overline{S}]|=r$ for $S\subseteq V(G)$, then both $S$ and $\overline{S}$ have at least $d+1$ vertices. Furthermore, if $r$ is odd, then both have at least $d+2$ vertices.
\end{prop}
\begin{proof}
If $\abs{S}\leq d$, then 
\[d-1\geq r \geq \abs{S}(d+1-\abs{S})\geq d,\]
which is a contradiction.
Now we assume that $r$ is odd.
Since $\abs{S}\geq d+1$, we may assume that $\abs{S}=d+1$. Then
\[d\abs{S}=d(d+1)=2|E(G[S])|+r.\]
Since $d(d+1)$ is even and $r$ is odd, this is a contradiction.
\end{proof}

Now, we are ready to show that Conjecture~\ref{O-conj} is true. We prove the contrapositive of Theorem~\ref{main}$:$ if $\kappa'(G) \le t$, then $\lambda_2(G) \ge \rho(d,t)$. Since $G_{d,t}$ is a $d$-regular graph with $\kappa'(G_{d,t})=t$ and $\lambda_2(G_{d,t})=\rho(d,t)$, the upper bounds for $\lambda_2(G)$ in Theorem~\ref{main} are sharp.

\begin{thm}\label{main}
Let $d$ and $t$ be positive integers such that $3 \le t \le d-1$.
If $G$ is a $d$-regular simple graph 
with $\lambda_2(G) < \frac{d-3+\sqrt{(d+3)^2-8t}}{2}$,
then $\kappa'(G) \ge t+1$.
Furthermore, for odd $t$,
if $G$ is a $d$-regular simple graph 
with $\lambda_2(G) < \frac{d-4+\sqrt{(d+4)^2-8t}}{2}$,
then $\kappa'(G) \ge t+1$.
\end{thm}
\begin{proof}
Assume to the contrary that $\kappa'(G) \leq t$.
Then there exists a vertex subset $S\subseteq V(G)$ 
such that $\abs{[S,\overline{S}]}=r\leq t$
(see Figure~\ref{2by2_quotient}).
\begin{figure}[h]
\centering
	\begin{tikzpicture}[auto,node distance=1cm,semithick]
		\node(a) {$S$};
		\draw(a) circle(1cm);
		\node(a1) [above of=a, xshift=3mm, yshift=-3mm] {};
		\node(a2) [below of=a1, yshift=8mm] {};
		\node(a3) [below of=a, xshift=3mm, yshift=3mm] {};
		\node(a4) [above of=a3, yshift=-8mm] {};

		\node(c) [right of=a, xshift=10mm] {$\vdots$};
		\node [below of=c] {$r$};
		
		\node(b) [right of=c, xshift=10mm] {$\overline{S}$};
		\draw(b) circle(1cm);
		\node(b1) [above of=b, xshift=-3mm, yshift=-3mm] {};
		\node(b2) [below of=b1, yshift=8mm] {};
		\node(b3) [below of=b, xshift=-3mm, yshift=3mm] {};
		\node(b4) [above of=b3, yshift=-8mm] {};

		\path[-]
	    (a1) edge  (b1)
	    (a2) edge  (b2)
	    (a3) edge  (b3)
	    (a4) edge  (b4);
	
	\end{tikzpicture}
	\caption{$|[S,\overline{S}]|=r$}
	\label{2by2_quotient}
\end{figure}

Let $s=\abs{S}$ and let $s'=\abs{\overline{S}}$.
Then the quotient matrix of the partition $S$ and $\overline{S}$ is
\[Q_0=\left(
\begin{array}{cc}
	d-\frac{r}{s} & \frac{r}{s} \\
	\frac{r}{s'}  & d-\frac{r}{s'}
\end{array} \right).\]
The eigenvalues of the matrix $Q_0$ are $d$ and $d-\frac{r}{s}-\frac{r}{s'}$.
By Corollary~\ref{quotient_eig_interlace}
\begin{eqnarray}\label{quotient_eigen}
	\lambda_2(G)\geq \lambda_2(Q_0) =d-\frac{r}{s}-\frac{r}{s'}.
\end{eqnarray}\\







{\it Case 1: $r \le t-1.$}
By~(\ref{quotient_eigen}) and Proposition~\ref{component_size},
\[\lambda_2(G) \geq d-\frac{2r}{d+1} \geq d-\frac{2(t-1)}{d+1}.\]
Since $\left[\frac{4(t-1)}{d+1}\right]^2-\frac{16(t-1)}{d+1}+8
=\left[\frac{4(t-1)}{d+1}-2\right]^2+4>0$, we have

\begin{eqnarray*}
	\left[\frac{4(t-1)}{d+1}\right]^2-\frac{16(t-1)}{d+1}+8-8t &>& -8t   \\
	\left[\frac{4(t-1)}{d+1}\right]^2-\frac{16(t-1)}{d+1}-\frac{8(t-1)(d+1)}{d+1}
	&>& -8t   \\
	\left[\frac{4(t-1)}{d+1}\right]^2-\frac{8(t-1)(d+3)}{d+1}+(d+3)^2
	&>& (d+3)^2-8t   \\
	d+3-\frac{4(t-1)}{d+1}
	&>& \sqrt{(d+3)^2-8t}  
\end{eqnarray*}

\[d-\frac{2(t-1)}{d+1}>\frac{d-3+\sqrt{(d+3)^2-8t}}{2}.\]\\

{\it Case 2: $r=t$.}
Consider the vertex partition $V(G)=V_1\cup V_2\cup V_3\cup V_4$ such that $S=V_1 \cup V_2$,  $\overline{S}=V_3\cup V_4$, and $V_2$ and $V_3$ are the endpoints of the two edges between $S$ and $\overline{S}$.
Let $\alpha=|V_2|$ and let $\beta=|V_3|$.
Then $|V_1|=s-\alpha$ and $|V_4|=s'-\beta$.
Let $k=|E(G[V_2])|$ and $l=|E(G[V_3])|$~(see Figure~\ref{4by4_quotient}).

\begin{figure}[h]
\centering
	\begin{tikzpicture}[auto,node distance=1cm,semithick]
		\node(a) {};
		\draw(a) circle(1cm);
		\node 	  [xshift=-4mm]		{$s-\alpha$};
		\node(a1) [above of=a, xshift=3mm, yshift=-3mm] {};
		\node(a2) [below of=a1, yshift=8mm] {};
		\node(a3) [below of=a, xshift=3mm, yshift=3mm] {};
		\node(a4) [above of=a3, yshift=-8mm] {};

		\node(c) [right of=a, xshift=10mm] {$\vdots$};
		\node 	 [below of=c] {$t$};
		
		\node(b) [right of=c, xshift=10mm] {};
		\draw(b) circle(1cm);
		\node 	 [right of=c, xshift=13.5mm] {$s'-\beta$};
		\node(b1) [above of=b, xshift=-3mm, yshift=-3mm] {};
		\node(b2) [below of=b1, yshift=8mm] {};
		\node(b3) [below of=b, xshift=-3mm, yshift=3mm] {};
		\node(b4) [above of=b3, yshift=-8mm] {};

		\path[-]
	    (a1) edge  (b1)
	    (a2) edge  (b2)
	    (a3) edge  (b3)
	    (a4) edge  (b4);

		\node(1) [above of=a, xshift=2.5mm, yshift=1mm] {};
		\node(2) [below of=a, xshift=2.5mm, yshift=-1mm] {};
		\node(3) [above of=b, xshift=-2.5mm, yshift=1mm] {};
	    \node(4) [below of=b, xshift=-2.5mm, yshift=-1mm] {};
	    \node 	 [right of=a, xshift=-3mm] {$\alpha$};
	    \node 	 [left of=b, xshift=3mm] {$\beta$};
	    \node 	 [above of=c, xshift=-13mm] {$k$};
	    \node 	 [above of=c, xshift=13mm] {$l$};

	    \path[-]
	    (1) edge (2)
	    (3) edge (4);
	
	\end{tikzpicture}
	\caption{$|[S,\overline{S}]|=t$}
	\label{4by4_quotient}
\end{figure}

The quotient matrix of the partition $V_1$, $V_2$, $V_3$ and $V_4$ is 
\[Q_1=\left(
\begin{array}{cccc}
	\frac{d(s-\alpha)-(d\alpha-2k-t)}{s-\alpha} & \frac{d\alpha-2k-t}{s-\alpha} & 0 					 & 0 \\
	\frac{d\alpha-2k-t}{\alpha} 			    & \frac{2k}{\alpha} 		  & \frac{t}{\alpha} 	 & 0 \\
	0 											& \frac{t}{\beta} 			  & \frac{2l}{\beta} 	 & \frac{d\beta-2l-t}{\beta}\\
	0 	& 0		& \frac{d\beta-2l-t}{s'-\beta}  & \frac{d(s'-\beta)-(d\beta-2l-t)}{s'-\beta}
\end{array} \right).\]
Note that $Q_1$ is non-negative tridiagonal matrix 
and that each row sum of $Q_1$ is $d$.
By Theorem~\ref{DRG-BCN}, we have
\[\widetilde{Q_1}=\left(
\begin{array}{ccc}
	\frac{2k+t}{\alpha}-\frac{d\alpha-2k-t}{s-\alpha} 	& \frac{t}{\alpha} 					 & 0 \\
	d-\frac{2k+t}{\alpha} 								& d-\frac{t}{\alpha}-\frac{t}{\beta} & d-\frac{2l+t}{\beta}\\
	0 													& \frac{t}{\beta} 					 & \frac{2l+t}{\beta}-\frac{d\beta-2l-t}{s'-\beta}
\end{array} \right)\]
and $\lambda_1(\widetilde{Q_1})=\lambda_2(Q_1)$.
Since $s\geq d+1$ by Proposition~\ref{component_size}, 
\[\frac{2k+t}{\alpha}-\frac{d\alpha-2k-t}{s-\alpha} 
	\geq \frac{2k+t}{\alpha}-\frac{d\alpha-2k-t}{d+1-\alpha}\]
and similarly
\[\frac{2l+t}{\beta}-\frac{d\beta-2l-t}{s'-\beta} 
	\geq \frac{2l+t}{\beta}-\frac{d\beta-2l-t}{d+1-\beta}.\]

Let
\[\widetilde{Q'_1}=\left(
\begin{array}{ccc}
	\frac{2k+t}{\alpha}-\frac{d\alpha-2k-t}{d+1-\alpha}  		  & \frac{t}{\alpha} 	& 0 \\
	d-\frac{2k+t}{\alpha} 	& d-\frac{t}{\alpha}-\frac{t}{\beta}  & d-\frac{2l+t}{\beta}\\
	0 						& \frac{t}{\beta} 			  	      & \frac{2l+t}{\beta}-\frac{d\beta-2l-t}{d+1-\beta}
\end{array} \right).\]
Then for large enough $\Delta\in\mathbb{R}$,
$\widetilde{Q_1}+\Delta I$ and $\widetilde{Q'_1}+\Delta I$ are non-negative matrices.
Since $\widetilde{Q_1}+\Delta I - (\widetilde{Q'_1}+\Delta I)$ is non-negative matrix,
by Perron-Frobenius,
\[\lambda_1(\widetilde{Q_1}+\Delta I)\geq \lambda_1(\widetilde{Q'_1}+\Delta I),\]
which implies that
\begin{equation}\label{even_by_PF}
	\lambda_1(\widetilde{Q_1})\geq \lambda_1(\widetilde{Q'_1}).
\end{equation}
Thus we consider $s=d+1$ in $\widetilde{Q_1}$. Then by the Degree-Sum formula,
\begin{eqnarray*}
	ds=d(d+1) &=& 2\set{\frac{d(d+1-\alpha)}{2}+\frac{\alpha(d+1-\alpha)}{2}+k}+t \\
		   &=& d(d+1)-\alpha d + \alpha d+\alpha(1-\alpha)+2k+t,
\end{eqnarray*}
which implies $k=\frac{\alpha(\alpha-1)}{2}-\frac{t}{2}$. Thus we have
\begin{eqnarray*}
	\frac{2k+t}{\alpha}-\frac{d\alpha-2k-t}{d+1-\alpha}
	&=& \frac{\alpha(\alpha-1)-t+t}{\alpha}
			-\frac{d\alpha-\alpha(\alpha-1)+t-t}{d+1-\alpha}\\
	&=& \alpha-1-\alpha=-1.
\end{eqnarray*}
Similarly, we have $l=\frac{\beta(\beta-1)}{2}-\frac{t}{2}$, thus
\[\widetilde{Q'_1}=\left(
\begin{array}{ccc}
	-1         & \frac{t}{\alpha} 				    & 0         \\
	d-\alpha+1 & d-\frac{t}{\alpha}-\frac{t}{\beta} & d-\beta+1 \\
	0 		   & \frac{t}{\beta} 			  	    & -1
\end{array} \right).\]

The characteristic polynomial of $\widetilde{Q'_1}$ is
\begin{eqnarray*}
	&&(x+1)\set{(x-d+\frac{t}{\alpha}+\frac{t}{\beta})(x+1)+(-d+\beta-1)\frac{t}{\beta}}
			+\frac{t}{\alpha}\set{(-d+\alpha-1)(x+1)}\\
	&&=(x+1)\set{x^2+(-d+1+\frac{t}{\alpha}+\frac{t}{\beta})x+(-d-\frac{dt}{\beta}-\frac{dt}{\alpha}+2t)}.
\end{eqnarray*}
The largest root of the characteristic polynomial is 
\begin{eqnarray}\label{even_largeeq}
\frac{d-1-\frac{t}{\alpha}-\frac{t}{\beta}+\sqrt{(d-1-\frac{t}{\alpha}-\frac{t}{\beta})^2
+4(d+\frac{dt}{\alpha}+\frac{dt}{\beta}-2t)}}{2}.
\end{eqnarray}
We show that ($\ref{even_largeeq}$) is a non-increasing function of $\alpha$.
Claim that 
\begin{eqnarray}\label{even_alpha_and_t}
	\nonumber
	&&\frac{d-1-\frac{t}{\alpha+1}-\frac{t}{\beta}+\sqrt{(d-1-\frac{t}{\alpha+1}-\frac{t}{\beta})^2
			+4(d+\frac{dt}{\alpha+1}+\frac{dt}{\beta}-2t)}}{2}\\
	&&\leq\frac{d-1-\frac{t}{\alpha}-\frac{t}{\beta}+\sqrt{(d-1-\frac{t}{\alpha}-\frac{t}{\beta})^2
			+4(d+\frac{dt}{\alpha}+\frac{dt}{\beta}-2t)}}{2}.
\end{eqnarray}
Since $\frac{t}{\alpha+1}-\frac{t}{\alpha}=\frac{-t}{\alpha(\alpha+1)}$, we want to show that
\begin{eqnarray*}
	0&\leq&\sqrt{(d-1-\frac{t}{\alpha}-\frac{t}{\beta})^2
			+4(d+\frac{dt}{\alpha}+\frac{dt}{\beta}-2t)} \\
	&&- \frac{t}{\alpha(\alpha+1)} - \sqrt{(d-1-\frac{t}{\alpha+1}-\frac{t}{\beta})^2
			+4(d+\frac{dt}{\alpha+1}+\frac{dt}{\beta}-2t)}.
\end{eqnarray*}
It suffices to show that
\begin{eqnarray}\label{eq5}
	\nonumber 0&\leq&(d-1-\frac{t}{\alpha}-\frac{t}{\beta})^2 +4(d+\frac{dt}{\alpha}+\frac{dt}{\beta}-2t)\\
	\nonumber &&- \left[\frac{t}{\alpha(\alpha+1)}\right]^2
			-2\left[\frac{t}{\alpha(\alpha+1)}\right]\sqrt{(d-1-\frac{t}{\alpha+1}-\frac{t}{\beta})^2
			+4(d+\frac{dt}{\alpha+1}+\frac{dt}{\beta}-2t)}\\
	\nonumber &&	 -(d-1-\frac{t}{\alpha+1}-\frac{t}{\beta})^2-4(d+\frac{dt}{\alpha+1}+\frac{dt}{\beta}-2t) \\
	 &=&(d-1-\frac{t}{\alpha}-\frac{t}{\beta})^2
			-(d-1-\frac{t}{\alpha+1}-\frac{t}{\beta})^2
			+4\frac{dt}{\alpha}-4\frac{dt}{\alpha+1}
			-\left[\frac{t}{\alpha(\alpha+1)}\right]^2 \\
 	\nonumber &&-2\left[\frac{t}{\alpha(\alpha+1)}\right]\sqrt{(d-1-\frac{t}{\alpha+1}-\frac{t}{\beta})^2
			+4(d+\frac{dt}{\alpha+1}+\frac{dt}{\beta}-2t)}.
\end{eqnarray}
By simplifying ($\ref{eq5}$), we have
\begin{eqnarray*}
	&&\set{(d-1-\frac{t}{\beta})-\frac{t}{\alpha}}^2
			-\set{(d-1-\frac{t}{\beta})-\frac{t}{\alpha+1}}^2
			+\frac{4dt}{\alpha(\alpha+1)}-\left[\frac{t}{\alpha(\alpha+1)}\right]^2	\\
	&&=-\frac{2t}{\alpha}(d-1-\frac{t}{\beta})+\frac{t^2}{\alpha^2}
			+\frac{2t}{\alpha+1}(d-1-\frac{t}{\beta})-\frac{t^2}{(\alpha+1)^2}
			+\frac{4dt}{\alpha(\alpha+1)}-\frac{t^2}{\alpha^2(\alpha+1)^2} \\
	&&=\set{\frac{2t}{\alpha+1}-\frac{2t}{\alpha}}(d-1-\frac{t}{\beta})
			+\frac{(\alpha+1)^2t^2-\alpha^2t^2}{\alpha^2(\alpha+1)^2}
			+\frac{4dt}{\alpha(\alpha+1)}-\frac{t^2}{\alpha^2(\alpha+1)^2} \\
	&&=\frac{-2t}{\alpha(\alpha+1)}(d-1-\frac{t}{\beta})+\frac{4dt}{\alpha(\alpha+1)}
			+\frac{(\alpha+1)^2t^2-\alpha^2t^2-t^2}{\alpha^2(\alpha+1)^2}   \\
	&&=\frac{-2t}{\alpha(\alpha+1)}(-d-1-\frac{t}{\beta})
			+\frac{2\alpha t^2}{\alpha^2(\alpha+1)^2}\\
	&&=\frac{-2t}{\alpha(\alpha+1)}(-d-1-\frac{t}{\beta}-\frac{t}{\alpha+1}).
\end{eqnarray*}
Thus we want to show that
\begin{eqnarray*}
	0&\leq&\frac{-2t}{\alpha(\alpha+1)}(-d-1-\frac{t}{\alpha+1}-\frac{t}{\beta})\\
	&&-2\left[\frac{t}{\alpha(\alpha+1)}\right]\sqrt{(d-1-\frac{t}{\alpha+1}-\frac{t}{\beta})^2
			+4(d+\frac{dt}{\alpha+1}+\frac{dt}{\beta}-2t)},
\end{eqnarray*}
which suffices to show that
\begin{eqnarray*}
	0&\leq&\set{\frac{2t}{\alpha(\alpha+1)}}^2(d+1+\frac{t}{\alpha+1}+\frac{t}{\beta})^2\\
	&&-\set{\frac{2t}{\alpha(\alpha+1)}}^2
			\set{(d-1-\frac{t}{\alpha+1}-\frac{t}{\beta})^2
			+4(d+\frac{dt}{\alpha+1}+\frac{dt}{\beta}-2t)}\\
	&=&\set{\frac{2t}{\alpha(\alpha+1)}}^2
			\set{(d+1+\frac{t}{\alpha+1}+\frac{t}{\beta})^2
			-(d-1-\frac{t}{\alpha+1}-\frac{t}{\beta})^2
			-4(d+\frac{dt}{\alpha+1}+\frac{dt}{\beta}-2t)}.
\end{eqnarray*}
Note that
\begin{eqnarray*}
	&&(d+1+\frac{t}{\alpha+1}+\frac{t}{\beta})^2 
			-(d-1-\frac{t}{\alpha+1}-\frac{t}{\beta})^2
			-4(d+\frac{dt}{\alpha+1}+\frac{dt}{\beta}-2t) \\
	&&=\set{d+(1+\frac{t}{\alpha+1}+\frac{t}{\beta})}^2 
			-\set{d-(1+\frac{t}{\alpha+1}+\frac{t}{\beta})}^2
			-4(d+\frac{dt}{\alpha+1}+\frac{dt}{\beta}-2t) \\
	&&=2d(1+\frac{t}{\alpha+1}+\frac{t}{\beta}) 
			+2d(1+\frac{t}{\alpha+1}+\frac{t}{\beta})
			-4(d+\frac{dt}{\alpha+1}+\frac{dt}{\beta}-2t) \\
	&&=4d(1+\frac{t}{\alpha+1}+\frac{t}{\beta}) 
			-4(d+\frac{dt}{\alpha+1}+\frac{dt}{\beta}-2t) \\
	&&=8t. 
\end{eqnarray*}
Thus Inequality~(\ref{even_alpha_and_t}) holds.

Similarly, we can show that ($\ref{even_largeeq}$) is a non-increasing function of $\beta$.

Since $\alpha, \beta \leq t$, if we replace both $\alpha$ and $\beta$ with $t$ in $\widetilde{Q'_1}$, then we have 
\[\widetilde{Q''_1}=\left(
\begin{array}{ccc}
	-1    & 1   & 0         \\
	d-t+1 & d-2 & d-t+1 \\
	0 	  & 1   & -1
\end{array} \right).\]
Then since ($\ref{even_largeeq}$) is a non-increasing function of $\alpha$ and $\beta$,  we have 
\[\lambda_2(G) \geq \lambda_2(Q_1)=\lambda_1(\widetilde{Q_1})\geq\lambda_1(\widetilde{Q'_1})\geq\lambda_1(\widetilde{Q''_1})=\frac{d-3+\sqrt{(d+3)^2-8t}}{2},\]
by Corollary~\ref{quotient_eig_interlace}, Theorem~\ref{DRG-BCN},and~(\ref{even_by_PF}). \\

In the above two cases, we showed that if $G$ is a $d$-regular graph with $\kappa'(G) \le t$, then
$\lambda_2(G) \ge \frac{d-3+\sqrt{(d+3)^2-8t}}2$, which proved the first statement in Theorem~\ref{main}. 
This result shows that when $r \le t-1$
\begin{eqnarray*}
	\lambda_2(G) \geq \frac{d-3+\sqrt{(d+3)^2-8(t-1)}}{2}
	> \frac{d-4+\sqrt{(d+4)^2-8t}}{2}.
\end{eqnarray*}
The remaining case for odd $t$ is  when $r=t$.\\

{\it Case 3: $r=t$ and $t$ is odd.} 
Consider the quotient matrix $Q_1$ in Case 2. Since $t$ is odd, we have $s,s' \ge d+2$ by Proposition~\ref{component_size}. We replace $s$ and $s'$ with $d+2$ to have 
\[\widetilde{Q'_1}=\left(
\begin{array}{ccc}
	\frac{2k+t}{\alpha}-\frac{d\alpha-2k-t}{d+2-\alpha} & \frac{t}{\alpha} 				    & 0 \\
	d-\frac{2k+t}{\alpha} 	 & d-\frac{t}{\alpha}-\frac{t}{\beta}  & d-\frac{2l+t}{\beta}\\
	0 						 & \frac{t}{\beta} 			  	    & \frac{2l+t}{\beta}-\frac{d\beta-2l-t}{d+2-\beta}
\end{array} \right),\]
and with a similar proof to Case 2, we have
\begin{equation}\label{odd_by_PF}
	\lambda_2(Q_1)=\lambda_1(\widetilde{Q_1})\geq \lambda_1(\widetilde{Q'_1}).
\end{equation}
Now, assume that both $S$ and $\overline{S}$ has exactly $d+2$ vertices.
Note that since every vertex has the fixed degree $d$, the number $k$ is determined by the number of edges between $V_1$ and $V_2$. For $v\in V_1$, there are only two choices: $v$ is adjacent to either $\alpha$ vertices in $V_2$ or $\alpha -1$ vertices
since $G$ is $d$-regular and $|S|=d+2$.

\noindent
Note that if every vertex in $V_1$ is adjacent to $\alpha$ vertices in $V_2$, then we have the minimum $k$ and that if every vertex in $V_1$ is adjacent to $\alpha-1$ vertices in $V_2$, then we have the maximum $k$. Thus we have the following inequalities:
\begin{eqnarray*}
	ds=d(d+2) &\leq& 2\set{\frac{d(d+2-\alpha)}{2}+\frac{\alpha(d+2-\alpha)}{2}+k}+t \\
		   	  &=&    d(d+2)-\alpha d + \alpha d+\alpha(2-\alpha)+2k+t.
\end{eqnarray*}
\begin{eqnarray*}
	ds=d(d+2) &\geq& 2\set{\frac{d(d+2-\alpha)}{2}+\frac{\alpha(d+2-\alpha-1)}{2}+k}+t \\
		      &=&    d(d+2)-\alpha d + \alpha d+\alpha(1-\alpha)+2k+t.
\end{eqnarray*}
By solving for $k$ in the above inequalities, we have
$\frac{\alpha(\alpha-1)}{2}-\frac{t}{2}-\frac{\alpha}{2}\leq k \leq \frac{\alpha(\alpha-1)}{2}-\frac{t}{2}$.
By setting $k=\frac{\alpha(\alpha-1)}{2}-\frac{t}{2}-\frac{\alpha}{2}+\epsilon$,
where $0\leq\epsilon\leq\frac{\alpha}{2}$, we have
\begin{eqnarray*}
 	\frac{2k+t}{\alpha}-\frac{d\alpha-2k-t}{d+2-\alpha}
 	&=& \frac{\alpha(\alpha-1)-t-\alpha+2\epsilon+t}{\alpha}
 			-\frac{d\alpha-\alpha(\alpha-1)+t+\alpha-2\epsilon-t}{d+2-\alpha}\\
 	&=& \alpha-2+\frac{2\epsilon}{\alpha}-\alpha+\frac{2\epsilon}{d+2-\alpha}
 		=-2+\frac{2\epsilon}{\alpha}+\frac{2\epsilon}{d+2-\alpha}.
\end{eqnarray*} 
Similarly, we can have 
$\frac{\beta(\beta-1)}{2}-\frac{t}{2}-\frac{\beta}{2}\leq l \leq \frac{\beta(\beta-1)}{2}-\frac{t}{2}$.
By setting $l=\frac{\beta(\beta-1)}{2}-\frac{t}{2}-\frac{\beta}{2}+\epsilon'$,
where $0\leq\epsilon'\leq\frac{\beta}{2}$, we have

\[\widetilde{Q'_1}=\left(
\begin{array}{ccc}
	-2+\frac{2\epsilon}{\alpha}+\frac{2\epsilon}{d+2-\alpha} & \frac{t}{\alpha}   & 0 \\
	d-\alpha+2-\frac{2\epsilon}{\alpha}		& d-\frac{t}{\alpha}-\frac{t}{\beta}  & d-\beta+2-\frac{2\epsilon'}{\beta} \\
	0 				& \frac{t}{\beta} 	 	& -2+\frac{2\epsilon'}{\beta}+\frac{2\epsilon'}{d+2-\beta}
\end{array} \right),\]
Let $A=\frac{2\epsilon}{\alpha}$ and $B=\frac{2\epsilon'}{\beta}$. Then
\[\widetilde{Q'_1}=\left(
\begin{array}{ccc}
	-2+\frac{(d+2)A}{d+2-\alpha} 	& \frac{t}{\alpha}   					& 0 \\
	d-\alpha+2-A				& d-\frac{t}{\alpha}-\frac{t}{\beta}  	& d-\beta+2-B \\
	0 							& \frac{t}{\beta} 	 					& -2+\frac{(d+2)B}{d+2-\beta}
\end{array} \right).\]
The characteristic polynomial of $\widetilde{Q'_1}$ is

\begin{eqnarray*}
	f(x)&=& \set{x+2-\frac{(d+2)A}{d+2-\alpha}}
		 	\bigg[(x-d+\frac{t}{\alpha}+\frac{t}{\beta})\left(x+2-\frac{(d+2)B}{d+2-\beta}\right)
		 	-(d-\beta+2-B)\frac{t}{\beta}\bigg]\\
	     && -\frac{t}{\alpha}(d-\alpha+2-A)\left[x+2-\frac{(d+2)B}{d+2-\beta}\right].
\end{eqnarray*}
If $A=0$, then we have  
\begin{eqnarray*}
	g(x)&=& (x+2)
		 	\bigg[(x-d+\frac{t}{\alpha}+\frac{t}{\beta})\left(x+2-\frac{(d+2)B}{d+2-\beta}\right)
		 	-(d-\beta+2-B)\frac{t}{\beta}\bigg] \\
	     &&-\frac{t}{\alpha}(d-\alpha+2)\left[x+2-\frac{(d+2)B}{d+2-\beta}\right].
\end{eqnarray*}
To show that the largest root of $g(x)$ is at most the largest root of $f(x)$, we first show that the largest root of $g(x)$ is at least $d-2$.

Let $\theta$ be the largest root of the polynomial $g(x)$.
Since $g(x)$ is a cubic function with a positive leading coefficient, 
if $g(d-2)\leq0$, then 
\begin{eqnarray}\label{d-2}
	d-2\leq\theta.
\end{eqnarray}
Claim that $g(d-2)\leq0$.
\begin{eqnarray*}
	& & g(d-2) \\
	&=& d \set{(-2+\frac{t}{\alpha}+\frac{t}{\beta})\left(d-\frac{(d+2)B}{d+2-\beta}\right)-(d-\beta+2-B)\frac{t}{\beta}}
				-\frac{t}{\alpha}(d-\alpha+2)\left[d-\frac{(d+2)B}{d+2-\beta}\right]\\
		   &=& d \bigg[-2d+(\frac{t}{\alpha}+\frac{t}{\beta})d
		   			+\frac{2(d+2)B}{d+2-\beta}-(\frac{t}{\alpha}+\frac{t}{\beta})\frac{(d+2)B}{d+2-\beta}
		   	 		-\frac{dt}{\beta}+t-\frac{2t}{\beta}+\frac{tB}{\beta}\bigg]\\
		   	&& -\frac{t}{\alpha}\bigg[d^2-\frac{d(d+2)B}{d+2-\beta}-\alpha d+\frac{\alpha(d+2)B}{d+2-\beta}
		   			+2d-\frac{2(d+2)B}{d+2-\beta}\bigg]\\
		   &=& d \bigg[-2d+\frac{2(d+2)B}{d+2-\beta}-\frac{t}{\beta}\frac{(d+2)B}{d+2-\beta}
		   	 		+t-\frac{2t}{\beta}+\frac{tB}{\beta}\bigg]\\
		   	&& -\frac{t}{\alpha}\bigg[-\alpha d+\frac{\alpha(d+2)B}{d+2-\beta}
		   			+2d-\frac{2(d+2)B}{d+2-\beta}\bigg]\\
		   &=& -2d^2+(2-\frac{2}{\beta}+\frac{B}{\beta}-\frac{2}{\alpha})dt
		   			+(2d-\frac{dt}{\beta}-t+\frac{2t}{\alpha})\frac{(d+2)B}{d+2-\beta}.
\end{eqnarray*}
If $2d-\frac{dt}{\beta}-t+\frac{2t}{\alpha}<0$, then
\begin{eqnarray*}
	g(d-2) &<& -2d^2+(2-\frac{2}{\beta}+\frac{B}{\beta}-\frac{2}{\alpha})dt.
\end{eqnarray*}
Since $0\leq B\leq1$ and $\alpha,\beta\leq t < d$, 
\begin{eqnarray*}
	g(d-2) &<& -2d^2+(2-\frac{2}{\beta}+\frac{1}{\beta}-\frac{2}{\alpha})dt=-2d^2+(2-\frac{1}{\beta}-\frac{2}{\alpha})dt\\
		   &<& -2d^2+(2-\frac{1}{t}-\frac{2}{t})dt=-2d^2+(2-\frac{3}{t})dt=d(-2d+2t-3) <0.
\end{eqnarray*}
Now, we assume that $2d-\frac{dt}{\beta}-t+\frac{2t}{\alpha}\geq 0$.
Since $0\leq B\leq1$ and $\beta\leq t$, 
\begin{eqnarray*}
	g(d-2) &\leq& -2d^2+(2-\frac{2}{\beta}+\frac{1}{\beta}-\frac{2}{\alpha})dt
					+(2d-\frac{dt}{\beta}-t+\frac{2t}{\alpha})\frac{d+2}{d+2-\beta}\\
		   &\leq& -2d^2+(2-\frac{1}{t}-\frac{2}{\alpha})dt+(2d-d-t+\frac{2t}{\alpha})\frac{d+2}{d+2-t}\\
		   &=&    -2d^2+2dt-d+(d-t)\frac{d+2}{d+2-t}+\frac{2t}{\alpha}(-d+\frac{d+2}{d+2-t}).
\end{eqnarray*}
Since $-d+\frac{d+2}{d+2-t}<0$ and $\alpha\leq t$,
\begin{eqnarray*}
	g(d-2) &\leq& -2d^2+2dt-d+(d-t)\frac{d+2}{d+2-t}+\frac{2t}{t}(-d+\frac{d+2}{d+2-t})\\
		   &=& 	  -2d^2+2dt-d+(d-t)\frac{d+2}{d+2-t}-2d+\frac{2(d+2)}{d+2-t}\\
		   &=& 	  -2d^2+(2t-3)d+(d-t+2)\frac{d+2}{d+2-t}.
\end{eqnarray*}
Since $t\leq d-1$, we have
\begin{eqnarray*}
	g(d-2) &\leq& -2d^2+(2d-5)d+d+2=-4d+2<0.
\end{eqnarray*}
Since $\theta$ is a root of $g(x)$, we have
\begin{eqnarray}\label{theta}
	\nonumber
	g(\theta) &=& (\theta+2)\set{(\theta-d+\frac{t}{\alpha}
					+\frac{t}{\beta})\left(\theta+2-\frac{(d+2)B}{d+2-\beta}\right)-(d-\beta+2-B)\frac{t}{\beta}} \\
			  && -\frac{t}{\alpha}(d-\alpha+2)\left[\theta+2-\frac{(d+2)B}{d+2-\beta}\right]=0.
\end{eqnarray}
Then from (\ref{theta}), we have
\begin{eqnarray}\label{theta_equal}
	\nonumber 
	&&(\theta-d+\frac{t}{\alpha}+\frac{t}{\beta})(\theta+2
	-\frac{(d+2)B}{d+2-\beta})-(d-\beta+2-B)\frac{t}{\beta}\hskip2cm \\
	&&= \frac{t}{\alpha(\theta+2)}(d-\alpha+2)\left[\theta+2-\frac{(d+2)B}{d+2-\beta}\right].
\end{eqnarray}
By applying (\ref{theta}) and (\ref{theta_equal}) to $f(\theta)$, we have
\begin{eqnarray*}
	f(\theta) &=& -\frac{(d+2)A}{d+2-\alpha}\set{\frac{t}{\alpha(\theta+2)}(d-\alpha+2)(\theta+2-\frac{(d+2)B}{d+2-\beta})}
			  		+\frac{At}{\alpha}(\theta+2-\frac{(d+2)B}{d+2-\beta})\\
	          &=&  \frac{At}{\alpha}(\theta+2-\frac{(d+2)B}{d+2-\beta})(-\frac{d+2}{\theta+2}+1).
\end{eqnarray*}
Then by~(\ref{d-2}), we have
\[\frac{At}{\alpha}(\theta+2-\frac{(d+2)B}{d+2-\beta})(-\frac{d+2}{\theta+2}+1)\le 0.\]
This implies that 
for any $A\neq0$, the largest eigenvalue of $\widetilde{Q'_1}$ is bigger than $\theta$~(see Figure~\ref{cubic_nonzero}).

\begin{figure}[h]
\begin{center}
		\begin{tikzpicture}
			\draw[->] (2,0) -- (11,0) node[right] {$x$};
		
			\tikzstyle{vertex}=[circle,fill,inner sep=1pt]
			\node[vertex] (a) at (9,0) {};
			\node[vertex] (b) at (9,-1.06) {};
		
			\draw (9.5,1.5) node[above, scale=0.8] {$A=0$};
			\draw (10.8,1.5) node[above, scale=0.8] {$A\neq0$};
			\draw (a) node[above, xshift=-1mm, scale=0.8] {$\theta$};
			\path[dotted,thick] (a) edge (b);
		
		    \draw[scale=3,domain=0.8:3.2,semithick,smooth,variable=\x] 
		    		plot(\x,{\x*\x*\x-6*\x*\x+11*\x-6});
		    \draw[scale=3,domain=2.8:3.5,semithick,smooth,variable=\x,dashed] 
		    		plot(\x,{(\x-0.3)*(\x-0.3)*(\x-0.3)-6*(\x-0.3)*(\x-0.3)+11*(\x-0.3)-6});
		\end{tikzpicture}
	\caption{}
	\label{cubic_nonzero}
\end{center}
\end{figure}

Similarly, we can show that for any $B\neq 0$, the largest eigenvalues of $\widetilde{Q_1'}$ is bigger than the largest root of $h(x)$, where $h(x)$ is the cubic polynomial obtained from $f(x)$ by plugging 0 into $B$.

Now, we assume that $A=0$ and $B=0$ in $\widetilde{Q_1'}$. Then we have 

\[\widetilde{Q'_1}=\left(
\begin{array}{ccc}
	-2	            & \frac{t}{\alpha}   					& 0         \\
	d-\alpha+2		& d-\frac{t}{\alpha}-\frac{t}{\beta}  	& d-\beta+2 \\
	0 				& \frac{t}{\beta} 	 					& -2
\end{array} \right).\]

The characteristic polynomial of $\widetilde{Q'_1}$ is 
\begin{eqnarray*}
	&&(x+2)\set{(x-d+\frac{t}{\alpha}+\frac{t}{\beta})(x+2)+(-d+\beta-2)\frac{t}{\beta}}
			+\frac{t}{\alpha}\set{(-d+\alpha-2)(x+2)}\\
	&&=(x+2)\set{x^2+(-d+\frac{t}{\alpha}+\frac{t}{\beta}+2)x+(-2d-\frac{dt}{\alpha}-\frac{dt}{\beta}+2t)}.
\end{eqnarray*}
The largest root of the characteristic polynomial is 
\begin{eqnarray}\label{odd_largeeq}
	\frac{d-2-\frac{t}{\alpha}-\frac{t}{\beta}+\sqrt{(d-2-\frac{t}{\alpha}-\frac{t}{\beta})^2
	+4(2d+\frac{dt}{\alpha}+\frac{dt}{\beta}-2t)}}{2}.
\end{eqnarray}
%
%
We show that~(\ref{odd_largeeq}) is a non-increasing funcion of $\alpha$. Claim that
\begin{eqnarray}\label{odd_alpha_and_t}
	\nonumber
	&&\frac{d-2-\frac{t}{\alpha+1}-\frac{t}{\beta}+\sqrt{(d-2-\frac{t}{\alpha+1}-\frac{t}{\beta})^2
			+4(2d+\frac{dt}{\alpha+1}+\frac{dt}{\beta}-2t)}}{2}\\
	&&\leq\frac{d-2-\frac{t}{\alpha}-\frac{t}{\beta}+\sqrt{(d-2-\frac{t}{\alpha}-\frac{t}{\beta})^2
			+4(2d+\frac{dt}{\alpha}+\frac{dt}{\beta}-2t)}}{2}.
\end{eqnarray}
Since $\frac{t}{\alpha+1}-\frac{t}{\alpha}=\frac{-t}{\alpha(\alpha+1)}$, we want to show that
\begin{eqnarray*}
	0&\leq&\sqrt{(d-2-\frac{t}{\alpha}-\frac{t}{\beta})^2
			+4(2d+\frac{dt}{\alpha}+\frac{dt}{\beta}-2t)}\\
	&&-\frac{t}{\alpha(\alpha+1)}
			-\sqrt{(d-2-\frac{t}{\alpha+1}-\frac{t}{\beta})^2
			+4(2d+\frac{dt}{\alpha+1}+\frac{dt}{\beta}-2t)}.
\end{eqnarray*}
It suffices to show that
\begin{eqnarray}\label{eq12}
	\nonumber
	0&\leq&(d-2-\frac{t}{\alpha}-\frac{t}{\beta})^2
			+4(2d+\frac{dt}{\alpha}+\frac{dt}{\beta}-2t)\\
	\nonumber
	&&-\set{\frac{t}{\alpha(\alpha+1)}}^2
			-2\set{\frac{t}{\alpha(\alpha+1)}}
			\sqrt{(d-2-\frac{t}{\alpha+1}-\frac{t}{\beta})^2
			+4(2d+\frac{dt}{\alpha+1}+\frac{dt}{\beta}-2t)}\\
	\nonumber
	&&-(d-2-\frac{t}{\alpha+1}-\frac{t}{\beta})^2
			-4(2d+\frac{dt}{\alpha+1}+\frac{dt}{\beta}-2t)\\
	&=&(d-2-\frac{t}{\alpha}-\frac{t}{\beta})^2
			-(d-2-\frac{t}{\alpha+1}-\frac{t}{\beta})^2
			+4\frac{dt}{\alpha}-4\frac{dt}{\alpha+1}
			-\set{\frac{t}{\alpha(\alpha+1)}}^2\\
	\nonumber
	&&-2\set{\frac{t}{\alpha(\alpha+1)}}
			\sqrt{(d-2-\frac{t}{\alpha+1}-\frac{t}{\beta})^2
			+4(2d+\frac{dt}{\alpha+1}+\frac{dt}{\beta}-2t)}.
\end{eqnarray}
By simplifying~(\ref{eq12}), we have
\begin{eqnarray*}
	&&\set{(d-2-\frac{t}{\beta})-\frac{t}{\alpha}}^2
			-\set{(d-2-\frac{t}{\beta})-\frac{t}{\alpha+1}}^2
			+\frac{4dt}{\alpha(\alpha+1)}-\left[\frac{t}{\alpha(\alpha+1)}\right]^2	\\
	&&=-\frac{2t}{\alpha}(d-2-\frac{t}{\beta})+\frac{t^2}{\alpha^2}
			+\frac{2t}{\alpha+1}(d-2-\frac{t}{\beta})-\frac{t^2}{(\alpha+1)^2}
			+\frac{4dt}{\alpha(\alpha+1)}-\frac{t^2}{\alpha^2(\alpha+1)^2} \\
	&&=\set{\frac{2t}{\alpha+1}-\frac{2t}{\alpha}}(d-2-\frac{t}{\beta})
			+\frac{(\alpha+1)^2t^2-\alpha^2t^2}{\alpha^2(\alpha+1)^2}
			+\frac{4dt}{\alpha(\alpha+1)}-\frac{t^2}{\alpha^2(\alpha+1)^2} \\
	&&=\frac{-2t}{\alpha(\alpha+1)}(d-2-\frac{t}{\beta})+\frac{4dt}{\alpha(\alpha+1)}
			+\frac{(\alpha+1)^2t^2-\alpha^2t^2-t^2}{\alpha^2(\alpha+1)^2}   \\
	&&=\frac{-2t}{\alpha(\alpha+1)}(-d-2-\frac{t}{\beta})
			+\frac{2\alpha t^2}{\alpha^2(\alpha+1)^2}\\
	&&=\frac{-2t}{\alpha(\alpha+1)}(-d-2-\frac{t}{\beta}-\frac{t}{\alpha+1}).
\end{eqnarray*}
Thus we want to show that
\begin{eqnarray*}
	0&\leq&\frac{-2t}{\alpha(\alpha+1)}(-d-2-\frac{t}{\alpha+1}-\frac{t}{\beta})\\
	&&-2\left[\frac{t}{\alpha(\alpha+1)}\right]\sqrt{(d-2-\frac{t}{\alpha+1}-\frac{t}{\beta})^2
			+4(2d+\frac{dt}{\alpha+1}+\frac{dt}{\beta}-2t)},
\end{eqnarray*}
which suffices to show that
\begin{eqnarray*}
	0&\leq&\set{\frac{2t}{\alpha(\alpha+1)}}^2(d+2+\frac{t}{\alpha+1}+\frac{t}{\beta})^2\\
	&&-\set{\frac{2t}{\alpha(\alpha+1)}}^2
			\set{(d-2-\frac{t}{\alpha+1}-\frac{t}{\beta})^2
			+4(2d+\frac{dt}{\alpha+1}+\frac{dt}{\beta}-2t)}\\
	&=&\set{\frac{2t}{\alpha(\alpha+1)}}^2
			\set{(d+2+\frac{t}{\alpha+1}+\frac{t}{\beta})^2
			-(d-2-\frac{t}{\alpha+1}-\frac{t}{\beta})^2
			-4(2d+\frac{dt}{\alpha+1}+\frac{dt}{\beta}-2t)}.
\end{eqnarray*}
Note that
\begin{eqnarray*}
	&&(d+2+\frac{t}{\alpha+1}+\frac{t}{\beta})^2 
			-(d-2-\frac{t}{\alpha+1}-\frac{t}{\beta})^2
			-4(2d+\frac{dt}{\alpha+1}+\frac{dt}{\beta}-2t) \\
	&&=\set{d+(2+\frac{t}{\alpha+1}+\frac{t}{\beta})}^2 
			-\set{d-(2+\frac{t}{\alpha+1}+\frac{t}{\beta})}^2
			-4(2d+\frac{dt}{\alpha+1}+\frac{dt}{\beta}-2t) \\
	&&=2d(2+\frac{t}{\alpha+1}+\frac{t}{\beta}) 
			+2d(2+\frac{t}{\alpha+1}+\frac{t}{\beta})
			-4(2d+\frac{dt}{\alpha+1}+\frac{dt}{\beta}-2t) \\
	&&=4d(2+\frac{t}{\alpha+1}+\frac{t}{\beta}) 
			-4(2d+\frac{dt}{\alpha+1}+\frac{dt}{\beta}-2t) \\
	&&=8t. 
\end{eqnarray*}
Thus inequality~(\ref{odd_alpha_and_t}) holds.

Similarly, we can show that~(\ref{odd_largeeq}) is a non-increasing funcion of $\beta$.

Since $\alpha, \beta\leq t$, if we replace both $\alpha$ and $\beta$
with $t$ in $\widetilde{Q'_1}$, then we have

\[\widetilde{Q''_1}=\left(
\begin{array}{ccc}
	-2    & 1   & 0         \\
	d-t+2 & d-2 & d-t+2 \\
	0 	  & 1   & -2
\end{array} \right).\]
Then since ($\ref{even_largeeq}$) is a non-increasing function of $\alpha$ and $\beta$, we have
\[\lambda_2(G) \geq \lambda_2(Q_1)=\lambda_1(\widetilde{Q_1})\geq\lambda_1(\widetilde{Q'_1})\geq\lambda_1(\widetilde{Q''_1})=\frac{d-4+\sqrt{(d+4)^2-8t}}{2},\]
 by Corollary~\ref{quotient_eig_interlace}, Theorem~\ref{DRG-BCN},and~(\ref{odd_by_PF}).
This completes the proof of the second statement in Theorem~\ref{main}.
\end{proof}



\end{document}